\documentclass[11pt]{amsart}
\usepackage{amsthm} 
\usepackage{color}
\usepackage{graphicx, epsfig}
\usepackage{epstopdf}
\usepackage{amsmath, amssymb, latexsym, euscript, amscd}
\usepackage{url}
\usepackage[all]{xy}
\usepackage{psfrag}
\usepackage{mathrsfs}
\usepackage{caption, wrapfig, multirow, tabularx, mathrsfs,verbatim}
\usepackage{subcaption}
\usepackage{mathptmx}
\usepackage[]{overpic}
\usepackage[dvipsnames]{xcolor}
\usepackage{thmtools}
\usepackage[colorlinks=true, linkcolor=blue, citecolor = blue, urlcolor=blue]{hyperref}
\usepackage{thmtools}
\usepackage{thm-restate}
\usepackage[]{cleveref}
\usepackage{tikz-cd} 
\usepackage[margin=1.50in]{geometry}
\usepackage{comment}
\usepackage{xcolor}

\newtheorem{theorem}{Theorem}[section]
\newtheorem{lemma}[theorem]{Lemma}
 
\newtheorem{definition}[theorem]{Definition} 
\newtheorem{proposition}[theorem]{Proposition}

\theoremstyle{remark}

\def\Isom{\operatorname{Isom}}
\def\gap{\vspace{.3cm}\noindent}

\def\nb{$\bullet\ \ \ $}
\def\smallskip{\vspace{.15cm}}
\def\medskip{\vspace{.3cm}}

\def\Pcal{\mathcal P}

\def\Wcal{\mathcal W}

\def\Vcal{\mathcal V}

\def\Rcal{\mathcal R}

\def\SS{\mathbb S}
\def\bdy{\partial}
\def\RR{\mathbb R}
\def\ZZ{\mathbb Z}
\def\ZZ{\mathbb Z}
\def\QQ{\mathbb Q}

\def\NN{\mathbb N}

\def\EuG{\EuScript{G}}

\def\Aut{\operatorname{Aut}}

\def\interior{\operatorname{int}}

\def\bdy{\partial}

\def\SL{\operatorname{SL}}

\def\GL{\operatorname{GL}}

\def\SOL{\operatorname{Sol}}

\def\Isom{\operatorname{Isom}}
 
\def\wtimes{\widetilde{\times}} 

\def\trace{\textnormal{tr}}

\def\hol{\operatorname{hol}}
\def\tr{\operatorname{tr}}
\def\MCG{\textnormal{Mod}}
\def\Id{\operatorname{Id}}

\def\mon{\operatorname{mon}}
\def\torus{T}
\def\pgamma{p_{_\Gamma}}
\def\qwcal{q_{_\Wcal}}

\newcommand{\bv}{\left[\begin{array}{c}}
\newcommand{\ev}{\end{array}\right]}
\newcommand{\bbmat}{\begin{bmatrix}} 
\newcommand{\ebmat}{\end{bmatrix}}
\newcommand{\bmat}{\begin{matrix}} 
\newcommand{\emat}{\end{matrix}}
\newcommand{\bpmat}{\begin{pmatrix}} 
\newcommand{\epmat}{\end{pmatrix}}

\newcommand{\beq}[1]{\begin{equation}\label{#1}}
\def\eeq{\end{equation}}

\title{A Combinatorial Characterization of SOL 3-Manifolds}

\author{Daryl Cooper}
\address{Mougins, France}
\email{cooper@math.ucsb.edu}

\author{Leslie Mavrakis}
\address{Department of Mathematics\\University of Utah\\Salt Lake City, UT 84112}
\email{l.mavrakis@utah.edu}

\author{Priyam Patel}
\address{Department of Mathematics\\University of Utah\\Salt Lake City, UT 84112}
\email{patelp@math.utah.edu}

\date{}

\begin{document}
\maketitle
\begin{abstract} We show that there is a universal compact branched 3-manifold $\Wcal_{\SOL}$ such that a closed 3-manifold $M$ immerses into $\Wcal_{\SOL}$ if and only if $M$ admits a $\SOL$ structure. Equivalently, a closed 3-manifold is $\SOL$ if and only if it has a certain type of triangulation. The construction of $\Wcal_{\SOL}$ is based on a regular language that characterizes Sol manifolds.
\end{abstract}

\section{Introduction}

This is the second in a series of papers giving a combinatorial characterization
of those closed 3-manifolds that admit a particular Thurston geometry.
In \cite{CPT}, a family of compact $n$-manifolds is defined to be \emph{locally combinatorially defined (LCD)} if it consists of those  $n$-manifolds that admit a triangulation where every vertex has a neighborhood simplicially isomorphic to one of finitely many triangulated balls. In other words, a family of manifolds is LCD if it is characterized by a local property of the triangulations. A family is \emph{BM} if it consists of the compact $n$-manifolds that properly immerse into a fixed compact piecewise linear (PL) branched $n$-manifold $\Wcal$. In this sense, $\Wcal$ may be considered as a kind of universal object or a classifying space for the family.  The main result of \cite{CPT} connects these combinatorial and topological viewpoints and is the following.

\begin{theorem} \label{equiv thm}
A family of compact PL $n$-manifolds is LCD if and only if it is BM.
\end{theorem}

A \emph{Klein geometry} is a pair $(G,X)$ where $G$ is a Lie group that acts transitively on a manifold $X$ by real analytic maps. A manifold has a \emph{$(G,X)$-structure} if it admits an atlas whose  charts have codomain $X$ and whose transition maps are restrictions of elements of $G$.   Thurston introduced eight such geometries for 3-manifolds, one of which is called $\SOL$. A compact 3-manifold admits a $\SOL$ structure if and only if it is a torus bundle over the circle with Anosov monodromy or a quotient of such a manifold by a free involution. The main theorem of this paper is:

\begin{restatable}{theorem}{MainResult}\label{solbrmfdthm}
(a) The family of closed 3-manifolds that admit a $\SOL$ structure is LCD.\\
(b) The family of $\SOL$ torus bundles is LCD.
\end{restatable}

The method of proof for \Cref{solbrmfdthm} is to define a compact branched 3-manifold $\Wcal_{\SOL}$ that is a generalized torus bundle over a train track, as was done for the family of all 3-manifolds that are torus bundles over the circle in
\cite[Section 7]{CPT}. It is then shown that a compact manifold $M$ properly immerses into $\Wcal_{\SOL}$ if and only if $M$ admits a $\SOL$ structure. 

The key to this approach is to show there is a subset of the mapping class group of the torus defined by a \emph{regular language} that gives
all $\SOL$ manifolds. 
This language is then encoded using matrix labels on the edges of a train track, and this describes $\Wcal_{\SOL}$. A similar regular language and associated branched manifold is used to prove $(b)$. 

There is a technical problem in doing so, which involves unwanted immersions of some manifolds into $\Wcal_{\SOL}$, and this is resolved using a covering space trick.
The existence of the two regular languages follows from an algebraic characterization, \Cref{matrix lemma}, of conjugacy classes of Anosov elements of $\GL(2,\ZZ)$ and a similar result for $\SOL$ semi-bundles.

In addition, $\Wcal_{\SOL}$ has an affine structure that pulls back to any $3$-manifold immersed in it, and the immersions  are $\pi_1$-injective. This is not the case for some other Thurston geometries.

\Cref{sec: torus bundles} reviews the classification of torus bundles. 
\Cref{sec: train tracks to brnmfld} describes a construction of a branched 3-manifold determined by a train track with edges labeled by integer matrices with non-zero determinant. This is a kind of generalized torus bundle over the train track, but with singular fibers where two copies of a torus times an interval are glued together by a finite degree covering map on their boundaries.
The main result of that section, \Cref{thm: immersions from paths},  classifies immersions into such branched 3-manifolds. \Cref{SOLsec} reviews $\SOL$ geometry  and the classification of $\SOL$ semi-bundles, which is less well-known than the classification of $\SOL$ bundles. \Cref{Regsec} describes the regular languages used for $\SOL$ manifolds, and \Cref{MainThmsec} proves the main theorem.

This work is based on the PhD thesis of the second author \cite{MR4957404}, supervised by the third author with input from the first author. 

\gap 

\noindent\textbf{Acknowledgments:} The third author was partially supported by National Science Foundation CAREER Grant DMS–2046889. This work was also supported by the National Science Foundation Grant DMS-2424139, while the third author was in residence at and the second author was visiting the Simons Laufer Mathematical Sciences Institute. The first author thanks France for being so wonderful. 

\section{Torus Bundles}\label{sec: torus bundles}

The torus $\RR^2/\ZZ^2$ is denoted by $\torus$. A torus bundle over the circle is a closed 3-manifold $M$ together with a submersion $\rho:M\rightarrow \SS^1$ that has torus fibers. Let $\widetilde{M}=\torus\times\RR$. Given a homeomorphism  $\phi:\torus\to \torus$, define an automorphism of $\widetilde{M}$ by
 $\Phi(x,t)= (\phi(x),t+1)$. 
Let $M_\phi =\widetilde{M}/\langle\Phi\rangle$. Then $\widetilde{M}
$ is a covering space of $M_{\phi}$ and $\Phi$ is a covering transformation. Now, $\torus\times[0,1]$ is a fundamental domain for the action of $\Phi$ and
\beq{eqx}M_{\phi}=\left(\ \torus\times[0,1] \ \right)/(x,0)\sim(\phi(x),1).\eeq

The map $\rho_\phi:M_\phi\to \SS^1$ given by $\rho_\phi([x,t])=\exp(2\pi i t)$ makes $M_{\phi}$ into  a torus bundle over $\SS^1$ with \emph{monodromy} $\phi$.
Every torus bundle over $\SS^1$ arises in this way.
Using the identification $H_1(\torus)\equiv H_1(\torus\times 0)\equiv H_1(\widetilde{M})$ then $\Phi_*=\phi_*$. With the identifications used in \Cref{eqx} the monodromy of the bundle (not its inverse) gives the action on $H_1(\widetilde M)$.

Let $\MCG(\torus)$ denote the mapping class group of $\torus$. 
  Given a group $G$, let $[G]$ denote the set of conjugacy classes of elements of $G$. The bundle map $\rho_{\phi}$ does not determine coordinates on the fibers,
 so only the conjugacy class $[\phi]\in[\MCG(\torus)]$ is a well-defined invariant of $\rho_{\phi}$.

 Two torus bundles $\rho_{\phi}:M_{\phi}\rightarrow \SS^1$ and $\rho_{\psi}:M_{\psi}\rightarrow \SS^1$
are \emph{virtually equivalent} if there is a submersion $h:M_{\psi}\rightarrow M_{\phi}$ such that $\rho_{\phi}\circ h=\rho_{\psi}$.
Then $h$ is a covering space that only 
unwraps torus fibers but not the circle base space. The cover
 of $M_{\phi}$ given by $h$ is  uniquely determined by a lattice
 $\Gamma\subset H_1(\torus;\ZZ)$ such that $\phi_*(\Gamma)=\Gamma$. Virtual equivalence is an equivalence relation by \Cref{virtequiv} below. 

If $h$ is also injective, the bundles are \emph{equivalent}.
By \Cref{virtequiv} this happens if and only if $\phi$ and $\psi$ are isotopic to conjugate homeomorphisms. Note that a single $3$-manifold may fiber over $\SS^1$ in more than one
way. For example, reversing the projection to $\SS^1$ changes $\phi$ to
$\phi^{-1}$.

Using the standard basis of $\RR^2$, an invertible $2\times2$  matrix $A$ determines
a linear automorphism of $\RR^2$. In what follows, the image in $\torus=\RR^2/\ZZ^2$ of the point $z=(x,y)\in \RR^2$ is denoted by  $[z]$.
If $A$ has integer entries, then it covers a map $\tau_A:\torus\rightarrow \torus$ given by 
 $\tau_A[z]=[Az]$ because $A(\ZZ^2)\subset\ZZ^2$. This is a local homeomorphism of degree $\det A$, and thus, the degree of $\tau_A$ as a covering map is $|\det A|$. 
 
The standard basis for $\RR^2$ determines a pair of oriented loops $\alpha, \beta$ in $\torus$ and $([\alpha], [\beta])$ is \emph{the standard basis} of $H_1(\torus; \ZZ)$. The induced map $(\tau_A)_*$ on $H_1(\torus; \ZZ)$ is represented by the matrix $A$ in this basis. In particular, when the monodromy $\phi=\tau_A$, the matrix of $\phi_*$ is $A$ using the basis above.
 Below, $M_A$ is the torus bundle over $\SS^1$ with monodromy $\tau_A$.

\begin{proposition}\label{virtequiv}
    If $A, B \in \GL(2, \ZZ)$, then the torus bundles $M_A$ and $M_B$ are\\
    \nb equivalent $\hspace{0.55 in}\qquad\Leftrightarrow \qquad [A]=[B]\in [\GL(2,\ZZ)].$\\
    \nb virtually equivalent $\qquad\Leftrightarrow\qquad[A]=[B]\in[\GL(2,\QQ)]$.
\end{proposition} 
\begin{proof}  If $[A]=[B] \in \GL(2, \QQ)$, then there is $C\in \GL(2,\QQ)$ with $CA=BC$. Multiplying by an integer we may assume $C$ has integer entries. Then $\Gamma=C(\ZZ^2)$ is a lattice
in $\ZZ^2$ of index $|\det C|$ and $CA(\ZZ^2)=BC(\ZZ^2)$. Now $A(\ZZ^2)=\ZZ^2$  so  that $\Gamma=B(\Gamma)$. The map $h:M_A\rightarrow M_B$ given by $h[x,t]=[\tau_{_C}x,t]$ shows that
 $M_A$ is the covering space of $M_B$
given by $\Gamma$. Thus, $M_A$ and $M_B$ are virtually equivalent. In the case that $C\in\
\GL(2,\ZZ)$ then $|\det C|=1$ so $h$ is a homeomorphism, which proves that $M_A$ and $M_B$ are equivalent. 

Conversely, if $h:M_A\rightarrow M_B$ is a virtual equivalence, then it induces an isomorphism on $H_1(\torus;\QQ)$ where the fibers over the basepoint of $M_A$ and $M_B$ are identified with $\torus$.
This implies $A$ and $B$ are conjugate in $\GL(2,\QQ)$. If $h$ is injective as well, then $h$ is a homeomorphism on the fibers and induces an isomorphism on $H_1(\torus; \ZZ)$. Thus, $A$ and $B$ are conjugate in $\GL(2, \ZZ)$. \end{proof}
    
Virtual equivalence is strictly weaker than equivalence. For example, all non-negative unipotent matrices in $\SL(2,\ZZ)$ except the identity  give inequivalent bundles that are virtually equivalent. 

\section{Branched Torus Bundles over Train Tracks}\label{sec: train tracks to brnmfld}

 A \emph{PL branched $n$-manifold} is defined in \cite{CPT} as a pair $(\Wcal,\Pi)$ where $\Wcal$ is a polyhedron and $\Pi$ is an atlas of charts satisfying certain properties. This is similar to, but different from, the definition of a smooth branched manifold originally defined by Williams \cite{MR348794}.  Informally, $\Wcal$ is locally the union of finitely many copies of $\RR^n$, called {\em sheets}, embedded in some high-dimensional Euclidean space so that sheets are tangent whenever they intersect. For simplicity of notation, the atlas $\Pi$ is omitted in this paper and a branched manifold is denoted by $\Wcal$.  Recall that a PL map from a manifold $M$ to a branched manifold $\Wcal$ is an \emph{immersion} if it is locally injective into sheets (see \cite[Definition 5.3]{CPT}).

The following is a generalization of torus bundles over 1-manifolds to the setting of branched manifolds. Train tracks are  PL branched 1-manifolds. The product of a torus $\torus$ and a train track $\Gamma$ is a branched 3-manifold. If $\Gamma$ has two vertices $a$ and $b$  of valence 1 then gluing $\torus\times a$ to $\torus\times b$ in $\torus\times \Gamma$  using a map of finite degree gives another branched 3-manifold by the Gluing Theorem \cite[Theorem 5.13]{CPT}. This procedure can be repeated with other vertices and motivates the following discussion.

A \emph{directed train track} is a train track where each edge has an orientation. The orientation of each edge is independent of the orientations of the other edges. 

\begin{definition}
A \emph{labeled train track} is a directed train track equipped with a map from the set of edges to $\mathcal{L} = \{ A \ | \ A\in M(2,\ZZ), \ \det A \neq 0\}\subset\GL(2,\QQ)$.
\end{definition}

Let $\mathbb{M} \subset\Gamma$ be the set of midpoints of edges and let $Y$ be the closure of a component of $\Gamma\setminus \mathbb{M}$. Then $Y$ is a star graph. Let $X$ be the disjoint union of all such $Y$.

The natural quotient map $q:X\rightarrow\Gamma$ is surjective,
and it is injective on $q^{-1}(\Gamma\setminus\mathbb{M})$.
If $x$ is the midpoint of an edge $e$ in $\Gamma$, then $q^{-1}(x)=\{x_-,x_+\}\subset X$ is
a pair of points, labeled so that $e$ is oriented from the $x_-$ side to $x_+$ side. Let $\bdy_{+}X\subset\bdy X$ be the subset of all $x_{+}$ points, and similarly define $\bdy_-X$.

 Each component $Y$ of $X$ is a branched 1-submanifold of $\Gamma$. Hence, there is a product branched manifold structure on $\torus\times X$ given by \cite[Theorem 5.10]{CPT}.
 Define $$f:=\torus\times\bdy_{-}X\rightarrow \torus\times \bdy_+X$$ as follows. Suppose $x$ is the midpoint of an edge $e$ of $\Gamma$ with $q^{-1}(x)=x_\pm$ and with edge label $A$. Define $f|:\torus\times x_{-}\rightarrow \torus\times x_{+}$ by $f([z],x_{-})=(\tau_A[z],x_+)$. Then, the Gluing Theorem \cite[Theorem 5.13]{CPT} gives a branched manifold $\Wcal_{\Gamma}$ obtained from $\torus\times X$ using the map $f$. Let 
 $$\qwcal:T\times X\longrightarrow\Wcal_{\Gamma}$$ be the associated quotient map. Additionally, the quotient map $\torus\times X\rightarrow\Gamma$  factors through a map $$\pgamma:\Wcal_{\Gamma}\rightarrow\Gamma$$ and the fiber $\pgamma^{-1}(x)$ at $x\in\Gamma$ is homeomorphic to $\torus$ and denoted by $T_x$.

\begin{definition}\label{def:train track mfd}
        $\Wcal_\Gamma$ is the \emph{branched 3-manifold
     associated to the labeled train track $\Gamma$}, and $\pgamma:\Wcal_{\Gamma}\rightarrow\Gamma$ is called a \emph{branched torus bundle over $\Gamma$}.
\end{definition}

An example is shown in \Cref{sol bundles} in \Cref{MainThmsec}. In such a diagram, if an edge does not have a label, then the label is $I$ (the identity matrix). If the edge does not have an orientation, then the label satisfies $A = A^{-1}$, and an arbitrary orientation can be assigned.

Observe that $\bdy\Wcal_{\Gamma}=\torus\times\bdy\Gamma$. The set $\Gamma\setminus\mathbb M$ is called the set of  {\em generic points}. A torus fiber in $\Wcal_{\Gamma}$ is {\em generic}
if it projects to a {\em generic point} of $\Gamma$.
The restriction of $\pgamma$ to the set of generic fibers is a torus fiber bundle, and $\pgamma$ is not a fiber bundle exactly above midpoints of edges labeled by $A$ where $|\det A|\ne 1$. If $x\in \Gamma$ is a generic point, then there is a unique $x'\in X$ with $q(x')=x$. The \emph{marking} at $x$ is a map $q_x$ that identifies a generic torus fiber $T_x\subset\Wcal_{\Gamma}$ with $T$ in the following  unique way $$q_x:T\rightarrow T_x\subset\Wcal_{\Gamma} \qquad \textnormal{ such that }\qquad q_x(z)=\qwcal(z,x').$$
The \emph{ marked basis} of $H_1(T_x;\ZZ)$ is the image of the standard basis of $H_1(T;\ZZ)$ under the induced map $(q_x)_*$. When $\Gamma$ is simply connected, then the image of this basis in $H_1(\Wcal_{\Gamma};\QQ)$ is called \emph{ the marked basis at $x$}. Throughout a basis is \emph{ordered}.

The following proposition shows that if a 3-manifold immerses into a branched torus bundle over a train track, then it is a torus bundle.

\begin{proposition}\label{prop:train track torus bundle}
    If $M$ is a compact connected 3-manifold and $\Omega: (M,\bdy M) \to (\Wcal_\Gamma,\bdy \Wcal_\Gamma)$ is a proper immersion, then there is a compact connected 1-manifold $C$ and a bundle map $\rho: M \to C$ with fiber a torus. Moreover, there is an immersion $\omega: C \to \Gamma$ such that the diagram below commutes, that is to say,  $\pgamma \circ \Omega = \omega \circ \rho$. 

        \[
\begin{tikzcd}
M \arrow{r}{\Omega} \arrow[swap]{d}{\rho} & \Wcal_\Gamma \arrow{d}{\pgamma} \\
C \arrow{r}{\omega} & \Gamma
\end{tikzcd}
\]
    \end{proposition}

\begin{proof}
    Suppose $\Omega: M \to \Wcal_\Gamma$ is an immersion. Set $H=\pgamma\circ\Omega$.
    Given $x \in M$ there is a neighborhood $U\subset M$ of $x$ such that $\Omega(U)$ is contained in a sheet of $\Wcal_\Gamma$, and in addition, there is a closed subset $I$ of $\Gamma$ such that $H(U)\subset I\subset \interior(I^+)$, where $I^{+}$ is a sheet of $\Gamma$. Choose $U$ small enough so that $I$ contains at most one vertex of $\Gamma$.

First, suppose that $I$ contains a vertex $v$ of $\Gamma$. Then, $v$ has a star  neighborhood  in $\Gamma$
    that is the union of finitely many arcs $I_1,\cdots, I_k$ which only intersect at $v$
    and contain no other vertices. Label these so that $I=I_1\cup I_2$. 

    Let $V_j$ be a connected component of $H^{-1}(I_j)$ for $j=1,2$. This is a closed subset of $M$.
    Now, $H$ is a PL submersion and $I_j\subset I\subset \interior(I^+)$ for $j = 1,2$,
    so it follows that $H$ is transverse to $\bdy I_j$, and therefore, $V_j$ is a compact 3-manifold with boundary $\bdy V_j=H^{-1}(\bdy I_j)$. 
    
    Since $H|_{V_j}:V_j\rightarrow I_j$ is a proper PL submersion between compact manifolds, it is a fiber bundle. 
    Each fiber first maps into a torus fiber of $\pgamma$ and, because $\Omega$ is an immersion, it finitely covers that torus. Therefore, $H|_{V_j}$ is a product bundle  $\torus\times I_j$. 
    
    If $I$ does not contain a vertex, then this argument simplifies but the same conclusion is obtained. Hence, each
    $x \in M$ has a neighborhood $V\cong \torus\times (I_1\cup I_2)\subset M$ and $H|_V:V\rightarrow I$ is a fiber bundle with fiber a torus. This gives  a torus bundle $\rho: M \to C$, where $C$ is a compact 1-manifold. Therefore, the immersion $\omega$ is defined by $\omega(c) = y$, where $y$ is the image of the torus $\rho^{-1}(c)$ under $H$, so that the diagram commutes. \end{proof}

\begin{definition}
    If $C\cong I$ in \Cref{prop:train track torus bundle},  the map $\omega$ is called the \emph{immersion path} for $\Omega$. If $C\cong \SS^1$, the map $\omega$ is called the \emph{immersion loop} for $\Omega$. 
\end{definition}

Next, we define the \emph{monodromy of a loop} and the \emph{holonomy along a path} in $\Gamma$, and relate these concepts to the monodromies of torus bundles that admit immersions into $\Wcal_\Gamma$.

 Let $\widetilde{\Gamma}$ be the universal cover of $\Gamma$.
 The map $\pgamma:\Wcal_\Gamma\rightarrow\Gamma$ is $\pi_1$-surjective.  Let $\Wcal_{\widetilde \Gamma}$ be the covering space  of $\Wcal_{\Gamma}$ given by $\ker \left(\pgamma\right)_*$. Now $\widetilde{\Gamma}$ is a labeled train track in an obvious way, and there is a projection $\widetilde{\pgamma}:\Wcal_{\widetilde\Gamma}\rightarrow\widetilde\Gamma$ that covers $\pgamma$. 
The branched manifold $\Wcal_{\widetilde{\Gamma}}$ is  the quotient
of copies of the product of $\torus$ with various trees, where boundary tori are glued by linear maps that induce isomorphisms on  rational homology. 
 It follows by
induction on the number of copies of $T\times ({\rm tree})$ and Mayer-Vietoris that 
$$\left(q_{\widetilde{x}}\right)_*:H_1(T;\QQ)\longrightarrow H_1(\Wcal_{\widetilde\Gamma};\QQ)$$
is an isomorphism. A covering transformation of $\widetilde\Gamma$ is uniquely covered by one of $\Wcal_{\widetilde\Gamma}$. Thus, there is an action of $\pi_1(\Gamma,x)$ induced by covering  transformations
  on $H_1(\Wcal_{\widetilde\Gamma};\QQ)$ that only depends on a choice of a basepoint $\widetilde x\in\widetilde\Gamma$ covering $x$.

 \begin{definition}\label{def:monodromy} The \emph{monodromy homomorphism} 
   $\mon:\pi_1(\Gamma,x)\longrightarrow\Aut H_1(\Wcal_{\widetilde\Gamma};\QQ)$
is $$\mon(\gamma)=\tau_*$$ where $\tau$ is the covering transformation of $\Wcal_{\widetilde \Gamma}$ given by $\gamma$. 
     \end{definition}

In the following, basepoints are chosen so that maps send basepoints to basepoints. 
For an immersion path $\omega:I\rightarrow\Gamma$, the basepoint is chosen to be $x=\omega(0)$. An immersion loop determines an element of $\pi_1(\Gamma,x)=\pi_1(\Gamma)$ denoted by  $\omega$.
Recall the definition of $M_{\phi}$, $\widetilde M$, and $\Phi$ in \Cref{sec: torus bundles} and the identification $\Phi_*=\phi_*$.

\begin{proposition}\label{MA immersion} Suppose $\Omega:M_{\phi}\rightarrow\Wcal_{\Gamma}$ is an immersion with immersion loop $\omega$. Then
$$\phi_*=(\widetilde\Omega)_*^{-1}\circ\mon(\omega)\circ(\widetilde\Omega)_*.$$
\end{proposition}
\begin{proof} Write $M=M_{\phi}$. There is a loop $\alpha:I\rightarrow  M$ so that $\rho_{\phi}\circ\alpha$ is a generator of $\pi_1
(\SS^1)$, and such that $\omega=\pgamma\circ\Omega\circ\alpha$ is an immersion. Let $\pi_M: \widetilde{M}\rightarrow M$ be the infinite cyclic cover given by unwrapping the circle base space of $M$. It follows that $\Omega$ is covered by an immersion $\widetilde{\Omega}:\widetilde{M}\rightarrow\Wcal_{\widetilde\Gamma}$. Moreover, $\widetilde{\Omega}_*$
is an isomorphism on rational homology. Define  $\Phi$ to be the covering transformation of $\widetilde{M}=T\times\RR$ given by $\Phi(x,t)=(\phi(x),t+1)$. 

 Choose a lift $\widetilde{\alpha}: I \to \widetilde{M}$ of $\alpha$ and let $\widetilde{\alpha}(0) =m$. Then, by the choice of $\alpha$ it follows that $\widetilde{\alpha}(1) = \Phi(m)$. Set $\widetilde{\omega}=\widetilde{\pgamma}\circ\widetilde{\Omega}\circ\widetilde{\alpha}.$ Then $\widetilde{\omega}:I\to\widetilde{\Gamma}$ is a lift of
$\omega=\pgamma\circ\Omega\circ\alpha$. Let $\tau'$ be the deck
transformation of $\widetilde{\Gamma}$ sending
$\widetilde{\omega}(0)$ to $\widetilde{\omega}(1)$, and let $\tau$
be the deck transformation of $\Wcal_{\widetilde{\Gamma}}$ covering
$\tau'$, that is to say, $\tau$ is the covering transformation given by $\omega$.  

  Now, $\widetilde{\Omega}\circ \widetilde{\alpha}$ is the lift of $\Omega \circ \alpha$ beginning at $\widetilde{\Omega}(m)$, and since $\omega=\pgamma\circ\Omega\circ\alpha$, the projection of $\widetilde{\Omega}\circ \widetilde{\alpha}$ to $\widetilde{\Gamma}$ is $\widetilde{\omega}$. Hence, the covering transformation $\tau$ sends the initial point of $\widetilde{\Omega}\circ \widetilde{\alpha}$ to its terminal point so that $\tau \circ \widetilde{\Omega}(m) = \widetilde{\Omega}\circ\Phi(m)$. Since $\tau \circ \widetilde{\Omega}$ and $\widetilde{\Omega}\circ\Phi$ are lifts of the same map $\Omega \circ \pi_M : \widetilde{M} \to \Wcal_\Gamma$ and $\widetilde{M}$ is connected, uniqueness of lifts implies that the following diagram commutes.
\[\begin{tikzcd}
\widetilde{M} \arrow{r}{\Phi} \arrow[swap]{d}{\widetilde{\Omega} }& \widetilde{M} \arrow{d}{\widetilde{\Omega}}\\
\Wcal_{\widetilde\Gamma} \arrow{r}{\tau} & \Wcal_{\widetilde\Gamma}
\end{tikzcd}\]

Now, $\widetilde{\Omega}_*:H_1(\widetilde M;\QQ)\rightarrow H_1(\Wcal_{\widetilde\Gamma};\QQ)$ is an isomorphism, and so $\Phi_*={\widetilde\Omega}_*^{-1}\circ\tau_*\circ{\widetilde{\Omega}}_*$.
\end{proof}

Suppose that $\widetilde{x}\in\widetilde{\Gamma}$ is a generic point. The marking $q_{\widetilde x}$ induces an isomorphism on rational homology
$H_1(T;\QQ)\rightarrow H_1(\Wcal_{\widetilde{\Gamma}};\QQ)$. 
Given another generic point $\widetilde{y}\in\widetilde{\Gamma}$ there is an automorphism
\beq{holonomy}\hol(\widetilde{x},\widetilde{y})=
\left(q_{\widetilde{y}}\right)_*\circ \left(q_{\widetilde{x}}\right)_*^{-1}\in\Aut H_1(\Wcal_{\widetilde{\Gamma}};\QQ),\eeq
which compares the local product structures in $\Wcal_{\widetilde{\Gamma}}$ at two different generic points. 
It sends the marked basis at $\widetilde x$ to the marked basis at $\widetilde y$.

\begin{definition}\label{def:holonomy}
    Suppose that  $\omega: [0,1] \to \Gamma$ is an immersed path between two generic points. Let $\widetilde{\omega}:[0,1]\rightarrow\widetilde{\Gamma}$ be the lift with $\widetilde{x}=\widetilde{\omega}(0)$ the basepoint of $\Wcal_{\widetilde\Gamma}$ and $\widetilde{y}=\widetilde{\omega}(1)$. The \emph{holonomy along $\omega$} is 
$\hol(\omega)=\hol(\widetilde{x},\widetilde{y})\in\Aut H_1(\Wcal_{\widetilde{\Gamma}};\QQ).$
\end{definition}

\Cref{lem: immersions from pathsa} expresses holonomy in terms of the edge labels $\omega$ traverses. Holonomy and monodromy agree for loops, however holonomy is defined more generally and enables the computation of monodromy.

\begin{definition}\label{def: pathprod}
Suppose that  $\omega: [0,1] \to \Gamma$ is an immersed path with generic endpoints, and $\omega^{-1}(\mathbb{M} )$ consists of points $m_1<m_2<\cdots<m_k$ where $\omega(m_i)$ is the midpoint of an edge $e_i\subset\Gamma$. If $A_i$ is the label on $e_i$, define $B_i=A_i^{\pm1}\in \GL(2,\QQ),$ where the sign is $1$ if and only if the direction the path crosses $\omega(m_i)$ is given by the orientation of $e_i$. The \emph{product of inverse matrices along $\omega$} is $A(\omega) = B_1^{-1}B_{2}^{-1}\cdots B_k^{-1}$. 
\end{definition}

If the linear map $F\in\Aut H_1(\Wcal_{\widetilde\Gamma};\QQ)$ and $\widetilde a,\widetilde{b}\in\widetilde\Gamma$, then $(F;\widetilde a,\widetilde b)$ is the matrix of $F$ using the marked basis at $\widetilde a$ in the domain and the marked basis at $\widetilde b$ in the codomain. 

\begin{lemma}\label{lem: immersions from pathsa}
     Suppose that  $\omega: [0,1] \to \Gamma$ is an immersed path between two generic points $x=\omega(0)$ and $y=\omega(1)$, and that $\widetilde{x}$ is a lift of $x$ to $\widetilde{\Gamma}$. Then
 \beq{matrixeqtn}  (\hol(\omega);\widetilde{x},\widetilde{x})=A(\omega).\eeq
 If $x=y$, then  $$\mon(\omega)= \hol(\omega).$$
\end{lemma}
\begin{proof} 
To prove \Cref{matrixeqtn}, let $k$ be as in \Cref{def: pathprod} and proceed by induction. First suppose that $k=1$ and the path crosses the midpoint of only one edge, $e$,  and in the same direction as the orientation of the edge. Then $B = B_1$ is the label on $e$ and $A =A(\omega)= B^{-1}$. Let $\widetilde{\omega}: I \to \widetilde{\Gamma}$ be a lift of $\omega$ that starts at $\widetilde{x}$ and ends at $\widetilde{y}$.
Let $([\alpha_1], [\beta_1])$ and $([\alpha_2], [\beta_2])$ be the marked bases at $\widetilde{x}$ and $\widetilde{y}$ of $H_1(\Wcal_{\widetilde\Gamma}; \QQ)$, respectively. The definition of the gluing map $B$ gives
$$[\alpha_1] = B[\alpha_2]  \quad \textnormal{ and } \quad [\beta_1] = B[ \beta_2].$$
Since $\hol(\widetilde x,\widetilde y)$ takes the marked basis at $\widetilde{x}$ to the marked basis at $\widetilde{y}$,
$$\hol(\widetilde{x},\widetilde{y})([\alpha_1]) = [\alpha_2]= B^{-1}[\alpha_1]\quad  \textnormal{ and }  \quad \hol(\widetilde{x},\widetilde{y})([\beta_1]) = [\beta_2]= B^{-1}[\beta_1].$$ Thus, $$(\hol(\widetilde{x},\widetilde{y});\widetilde x,\widetilde x)=B^{-1}= A.$$

If $\omega$ traverses $e$ in the negative direction, the result follows from switching the roles of $x$ and $y$. This completes the proof when $k=1$.

Inductively, suppose the result holds for all $n\leq k$ and that $\omega$ crosses $k+1$ midpoints. Let $\widetilde{x}$ and $\widetilde z$ be
the start and end of $\omega$ and $\widetilde{y}$ the endpoint of the penultimate edge.
Write $C=B_1^{-1}B_2^{-1}\cdots B_k^{-1}$ and $B=B_{k+1}^{-1}$.
By induction $(\hol(\widetilde{x},\widetilde{y});\widetilde x, \widetilde x)=C$ and
 $(\hol(\widetilde{y},\widetilde{z});\widetilde{y}, \widetilde{y}) = B$.  
The definition of holonomy in \Cref{holonomy} implies
$$\hol(\widetilde{x},\widetilde{z}) =\hol(\widetilde{y},\widetilde{z}) \circ \hol(\widetilde{x},\widetilde{y}).$$
Using $\cdot$ for matrix product, with respect to bases this becomes
\beq{eq5}(\hol(\widetilde{x},\widetilde{z});\widetilde{x},\widetilde x)  = (\hol(\widetilde{y},\widetilde{z});\widetilde{x},\widetilde x) \cdot (\hol(\widetilde{x},\widetilde{y});\widetilde{x},\widetilde x).\eeq
Computing $(\hol(\widetilde{y},\widetilde{z});\widetilde{x},\widetilde{x})$ 
involves the following change of basis
\beq{eq7}(\hol(\widetilde{y},\widetilde{z});\widetilde{x},\widetilde x) = (\Id;\widetilde y,\widetilde x)\cdot(\hol(\widetilde{y},\widetilde{z});\widetilde{y},\widetilde y)\cdot(\Id;\widetilde x,\widetilde y).\eeq
If a linear map $F$ sends the marked basis at $\widetilde x$ to the marked basis at $\widetilde y$, then
$$(F;\widetilde x,\widetilde x)=(\Id;\widetilde y,\widetilde x).$$
Applying this to $F=\hol(\widetilde x,\widetilde y)$ gives
\beq{keyeqtn}(\Id;\widetilde y,\widetilde x) =(\hol(\widetilde x,\widetilde y);\widetilde x,\widetilde x)= C.\eeq
Substituting into \Cref{eq7} gives
$$(\hol(\widetilde{y},\widetilde{z});\widetilde{x},\widetilde x) = (\hol(\widetilde x,\widetilde y);\widetilde x,\widetilde x)\cdot(\hol(\widetilde{y},\widetilde{z});\widetilde{y},\widetilde y)\cdot(\hol(\widetilde x,\widetilde y);\widetilde x,\widetilde x)^{-1}.
$$
Substituting this into \Cref{eq5}  gives

\gap
$$\begin{array}{rcl}
(\hol(\widetilde{x},\widetilde{z});\widetilde{x},\widetilde x) & = &
(\hol(\widetilde x,\widetilde y);\widetilde x,\widetilde x)\cdot(\hol(\widetilde{y},\widetilde{z});\widetilde{y},\widetilde y)\cdot(\hol(\widetilde x,\widetilde y);\widetilde x,\widetilde x)^{-1}\cdot (\hol(\widetilde{x},\widetilde{y});\widetilde{x},\widetilde x)\\
& = &(\hol(\widetilde x,\widetilde y);\widetilde x,\widetilde x)\cdot(\hol(\widetilde{y},\widetilde{z});\widetilde{y},\widetilde y)\\
&=& C\cdot B = B_1^{-1} \cdots B_{k}^{-1}B_{k+1}^{-1}=A(\omega).
\end{array}$$

\gap 
This concludes the proof of \Cref{matrixeqtn}.
For the second claim, suppose that $x=y$ and $\widetilde{\omega}:I\rightarrow\widetilde{\Gamma}$ covers $\omega$. Set $\widetilde{x}=\widetilde{\omega}(0)$ and $\widetilde{y}=\widetilde{\omega}(1)$.
Let $\tau$ be the covering transformation of $\Wcal_{\widetilde\Gamma}$ corresponding to $\widetilde\omega$. Then $id_{_T}=q_{\widetilde y}^{-1}\circ\tau\circ q_{\widetilde x}$ so

$$id=\left(q_{\widetilde y}\right)_*^{-1}\circ\tau_*\circ \left(q_{\widetilde x}\right)_*\in\Aut\ H_1(T;\QQ),$$
 and thus
$$\mon(\omega)=\tau_*=
(q_{\widetilde y})_*\circ (q_{\widetilde x})_*^{-1}=\hol(\omega)\in \Aut\ H_1(\Wcal_{\widetilde\Gamma};\QQ).$$ \end{proof}

Let $\pi:\Wcal_{\widetilde\Gamma}\rightarrow\Wcal_{\Gamma}$ be the covering space projection.  Then $q_x=\pi\circ q_{\widetilde x}$ where $\widetilde{x} \in \widetilde{\Gamma}$ maps to $x \in \Gamma$.
Let $T_x\subset\Wcal_{\Gamma}$ be the torus fiber above $x$.
Then $(q_x)_*H_1(T;\ZZ)$ is the image of $H_1(T_x;\ZZ)$ in $H_1(\Wcal_{\Gamma};\ZZ)$.

\begin{theorem}\label{thm: immersions from paths}
   (1)   Suppose that $M$ is a closed 3-manifold and $\Omega: M \to \Wcal_\Gamma$ is an immersion with immersion loop $\omega: \SS^1 \to \Gamma$. Then $M\cong M_A$ and $[(\mon(\omega);\widetilde x,\widetilde x)]= [A] \in [\GL(2, \QQ)].$ Thus, the only other manifolds that immerse into $\Wcal_{\Gamma}$ with immersion path $\omega$ are virtually equivalent to $M$.
   
(2) If $\omega$ is an immersion loop for $\Omega:M\rightarrow \Wcal_{\Gamma}$, then $\omega$ is also an immersion loop for every torus bundle virtually equivalent to $M$.

(3)(a) Suppose that $\omega: I \to \Gamma$ is an immersion with generic endpoints and $\omega(0) = x$.  Then
there exists $n>0$ and an immersion $\Omega: \torus \times I \to \Wcal_\Gamma$
with immersion path $\omega$ such that
   $$\Omega_*H_1(T\times 0;\ZZ)= n\cdot (q_x)_*H_1(T;\ZZ).$$
   (b) Let $\widetilde{\omega}:I\rightarrow\widetilde\Gamma$ be a lift of $\omega$ and $\widetilde{x}=\widetilde{\omega}(0)$. If $(\hol(\widetilde y,\widetilde{x});\widetilde{x},\widetilde{x})$ is  an integer matrix whenever $\widetilde y=\widetilde\omega(t)$ is generic,  then $n=1$ in $(a)$.

(4) Suppose that $\omega: \SS^1 \to \Gamma$ is an immersion and   $
  (\mon(\omega); \widetilde{x}, \widetilde{x})=A \in \GL(2,\ZZ)$. Then there is  an immersion $\Omega: M_A \to \Wcal_\Gamma$ with immersion loop $\omega$.

 \end{theorem}
\begin{proof} 
(1) By \Cref{prop:train track torus bundle}, $M\cong M_A$ for some
$A \in \GL(2,\ZZ)$, and the result follows by \Cref{MA immersion} and \Cref{virtequiv}.

\gap
(2) If $h:N\rightarrow M$ is a virtual equivalence, then 
$\Omega\circ h:N\rightarrow\Wcal_{\Gamma}$ is an immersion that also has  immersion loop $\omega$.

\gap 

 The proof of (3) and (4) uses subspaces of  covering spaces of $\Wcal_{\Gamma}$. These are branched manifolds associated to labeled graphs with projections $p_i:\Wcal_i\rightarrow\Gamma_i$. Each has 
a natural fiber-preserving  map $\Wcal_i\rightarrow\Wcal_{\Gamma}$. Either $\Wcal_{i}$ is a covering space of, or a subspace
of,  $\Wcal_{i-1}$.  Choose basepoints for fundamental groups $x_i\in\Gamma_i$ that project to $x\in\Gamma$
and $w_i\in\Wcal_i$ with $p_i(w_i)=x_i$. Start with  $\Gamma_0=\Gamma$, and $\Gamma_1=\widetilde{\Gamma}$, so $\Wcal_0=\Wcal_{\Gamma}$ and $\Wcal_1=\Wcal_{\widetilde\Gamma}$. 
\gap

\noindent (3a) Let $\omega_1:I\rightarrow\Gamma_1$ be the lift of $\omega$ with 
$\omega_1(0)=x_1$.
There is a minimal sequence of edges $J=\widetilde{e_1}\cdot\widetilde{e}_2\cdots\widetilde{e_k}\subset\Gamma_1$ that form an arc in $\Gamma_1$ that contains $\omega_1(I)$.
Then $J$ is a labeled train track. This determines the compact branched 3-manifold
 $\Vcal=p_1^{-1}(J)\subset\Wcal_1 = \Wcal_{\widetilde{\Gamma}}$. The covering map $\pi: \Wcal_{\widetilde{\Gamma}} \to \Wcal_{\Gamma}$ restricts to a map $\pi|:\Vcal\rightarrow\Wcal_{\Gamma}$.
 
 Let $\mathbb{M}'$ denote the set of midpoints of the edges in $J$. Let $Y_1,\cdots,Y_{k+1}$ be the arcs that are the closures of the  components of $J\setminus\mathbb{M}'$. 
 Set $V_i=\torus\times Y_i$. Then $\Vcal$ is the quotient of $\sqcup V_i$ by the gluing maps. 
  Let $G_i$ be  the image of $H_1(V_i;\ZZ)$ in $H_1(\Vcal;\ZZ)$.  The gluing maps induce isomorphisms on rational homology. Thus $G_i\cong H_1(V_i;\ZZ)$ has finite index in $H_1(\Vcal;\ZZ)$. Hence $\cap G_i$ also has finite index. There is $n>0$ such that $H=n\cdot G_1\subset \cap G_i\subset H_1(\Vcal;\ZZ)$.
  
 Using the Hurewicz homomorphism $$\theta:\pi_1(\Vcal)\rightarrow H_1(\Vcal;\ZZ)$$ there is a finite-sheeted cover $r:\widehat{\Vcal}\rightarrow\Vcal$ with $r_*(\pi_1\widehat{\Vcal})=\theta^{-1}(H)\subset\pi_1\Vcal$.   Because $H\subset G_i$, this cover is obtained by gluing copies of finite covers
of $V_i$ together by homeomorphisms between boundary components. Thus,
 $\widehat{\Vcal}$ is the branched manifold associated to some train track $\widehat J$, labeled with elements of $\GL(2, \ZZ)$. Letting $\widehat{p}: \widehat{\Vcal} \to \widehat{J}$ be the projection map, then $\widehat J$ is a finite tree.
 
 Now, $r$ sends torus fibers to torus fibers, so that it covers a map $\widehat J\rightarrow J$. Choose a lift $\widehat\omega$ of $\widetilde\omega$
to $\widehat J$. The image of $\widehat\omega$ is an embedded arc in $\widehat J$, also called $\widehat{\omega}$. Define  $\widehat{\Vcal}'=\widehat{p}\,^{-1}(\widehat{\omega})\subset\widehat{\Vcal}$. There is a homeomorphism
$\widehat{\Omega}:\torus\times I\rightarrow\widehat{\Vcal}'$ because the labels on $\widehat{J}$ are in $\GL(2,\ZZ)$. Then $\Omega=\pi\circ r\circ\widehat{\Omega}:T\times I\rightarrow\Wcal_{\Gamma}$  is an immersion with immersion path $\omega$ and $(r\circ\widehat{\Omega})_*H_1(T\times 0;\ZZ)=n\cdot G_1$. Now, \[G_1=(q_{\widetilde{x}})_*H_1(T;\ZZ)\subset H_1(\Vcal;\ZZ)\subset H_1(\Wcal_{\widetilde\Gamma};\ZZ) \] so that composing with $\pi_*$ and using
$q_x=\pi\circ q_{\widetilde x}$ proves (3a). 

\gap 

\noindent (3b) Observe that $\hol(\widetilde y,\widetilde x)=\hol(\widetilde x,\widetilde y)^{-1}$. Thus, the hypothesis  is equivalent to the statement that
for every $1\le i\le k$ the matrix $B_i\cdots B_1$ has all integer entries.

 Suppose that $\widetilde y=\widetilde{\omega}(t)$ is a generic point and $G_i=\left(q_{\widetilde y}\right)_*H_1(T;\ZZ).$ Then
$$\begin{array}{rrcl} &G_1&\subset &G_i\\
\Leftrightarrow & \left(q_{\widetilde x}\right)_*H_1(T;\ZZ) 
& \subset & \left(q_{\widetilde y}\right)_*H_1(T;\ZZ) \\
\Leftrightarrow & \left(q_{\widetilde y}\right)_*^{-1}\left(q_{\widetilde x}\right)_*H_1(T;\ZZ) 
& \subset & H_1(T;\ZZ) \\
\Leftrightarrow & \left(\left(q_{\widetilde x}\right)_*\left(q_{\widetilde y}\right)_*^{-1}\right)\left(q_{\widetilde x}\right)_*H_1(T;\ZZ) 
& \subset & \left(q_{\widetilde x}\right)_* H_1(T;\ZZ) \\
\Leftrightarrow & \hol(\widetilde y,\widetilde x)G_1 
& \subset & G_1 .
\end{array}$$
Using the marked basis at $\widetilde x$, the group $G_1$ 
is  the subgroup $\ZZ^2$ of $H_1(\Wcal_{\widetilde\Gamma};\QQ)=\QQ^2$. The condition that $\hol(\widetilde{y}, \widetilde{x})G_1 \subset G_1$ is then equivalent to
$(\hol(\widetilde y,\widetilde x);\widetilde x,\widetilde x)$ having integer entries for all generic points $\widetilde{y} = \widetilde{\omega}(t)$. Thus, the hypothesis implies that $G_1 \subset G_i$ for all $i$ and $n=1$.
\gap

\noindent (4)  Let $\Gamma_2\subset \Gamma_1$ be   the line that contains $x_1$ and is a component of the preimage of $\omega$. Let $\Wcal_2=p_1^{-1}(\Gamma_2)\subseteq \Wcal_1$. 
Choose the basepoint $x_2 = x_1$ for the fundamental group of $\Gamma_2$. There is a covering transformation
$\sigma_1$ of $\Gamma_1$ which preserves $\Gamma_2$ and corresponds
to $\omega$. Then $\sigma_2=(\sigma_1|_{\Gamma_2})$ preserves $\Gamma_2$. There is a covering transformation
$\tau_1$ of $\Wcal_1$ that covers $\sigma_1$ so that $\tau_2=(\tau_1|_{\Wcal_2})$ preserves $\Wcal_2$.

Observe that $\Vcal \subset \Wcal_2$. Consider the subgroup $H=n\cdot G_1\subset \cap G_i\subset H_1(\Vcal;\ZZ)$ from (3). Let $\EuScript{P}:\Wcal_3\rightarrow \Wcal_2$ be the cover  corresponding to the preimage of $H$ under the homomorphism $$\pi_1(\Wcal_2)\rightarrow H_1(\Wcal_2;\ZZ)\rightarrow H_1(\Wcal_2;\QQ).$$ 
If $T\subset\Wcal_3$ is a torus fiber, then $i_*:H_1(T;\ZZ)\rightarrow H_1(\Wcal_3;\ZZ)$ is an isomorphism. Thus, the graph $\Gamma_3$ is a tree with
 edge labels in $\GL(2,\ZZ)$, and $\Wcal_3\cong T\times\Gamma_3$. 

 Now, the action induced by $\tau_2$ on $H_1(\Wcal_2;\QQ)\cong\QQ^2$ is given by the matrix $A$ using the marked basis at $x_2$. Since $A\in\GL(2,\ZZ)$, it follows that $A$ preserves 
$$\ZZ^2\cong
 \left(q_{_{x_2}}\right)_*H_1(T;\ZZ)\subset H_1(\Wcal_2;\QQ).$$ Hence $A$ also preserves $H\cong n\cdot\ZZ^2$. Thus, $\tau_2\circ \EuScript{P}:\Wcal_3\rightarrow\Wcal_2$ has a lift to  a covering
transformation $\tau_3:\Wcal_3\rightarrow\Wcal_3$ that covers $\tau_2$. Further, $\tau_3$ covers an automorphism $\sigma_3$ of $\Gamma_3$ and $\sigma_3$ preserves a line $\Gamma_4\subset\Gamma_3$ that covers
$\omega$. Define $\Wcal_4=p_3^{-1}(\Gamma_4)$ and $\sigma_4=(\sigma_3|_{\Gamma_4})$. 

Then $\tau_4=(\tau_3|_{\Wcal_4})$ preserves $\Wcal_4$.
Let $\Omega: \torus \times I \to \Wcal_\Gamma$ be the immersion constructed in (3). The lift of $\Omega$ to $\Omega_4:T\times I\rightarrow \Wcal_4$ is injective. Set $X=\Omega_4(T\times I)\cong T\times I$. Then $X$ is
a fundamental domain for the action of $\tau_4$. Define $M=\Wcal_4/\langle\ \tau_4\ \rangle$. Now $\Wcal_4$ is a subspace of a covering space of $\Wcal_{\Gamma}$ and the covering space projection restricted to $\Wcal_4$ factors through an immersion $M \to \Wcal_{\Gamma}$ with immersion path $\omega$. 
 The result now follows from (1) and (2).
\end{proof}

\section{Sol Manifolds}\label{SOLsec} 

For background on $\SOL$ manifolds, see Martelli's book \cite{Martelli} and Scott's article \cite{MR705527}.
See also \cite{MR1886669}, \cite{MR4520303}, \cite{MR4614720}, and \cite{MR2632777}.
 $\SOL$ geometry is one of the eight Thurston geometries. It is $\mathbb{R}^3$ equipped with the Riemannian metric
\[
ds^2 = e^{2z}dx^2 + e^{-2z}dy^2 + dz^2.
\] By convention, the term $\SOL$ is used both for this metric space and for its full isometry group. A $3$-manifold $M$ admits a $\SOL$ structure if $M=\SOL/\Gamma$
where $\Gamma$ is a torsion-free discrete subgroup of $\Isom(\SOL)$.

A key feature of this geometry is that the identity component of the isometry group, denoted $\mathrm{Sol}_0$, acts simply transitively on $\mathbb{R}^3$. This gives an identification of $\mathrm{Sol}_0$ with $\mathbb{R}^3$, now viewed as a Lie group acting on itself by left multiplication. Under this identification, the group law is given by
\[
(x,y,z)\cdot(x',y',z') = \big(x + e^{-z}x',\; y + e^{z}y',\; z + z'\big).
\]
 
To understand the full isometry group, one must also consider isometries that fix a point. The stabilizer of $0 \in \mathbb{R}^3$ is a finite subgroup $G_0$ of $O(3)$, isomorphic to the dihedral group, $D_4$, of order $8$. It is generated by the maps
\[
(x,y,z) \mapsto (\pm x, \pm y, z), \qquad (x,y,z) \mapsto (y,x,-z).
\]

Combining these observations, every isometry of Sol can be written uniquely as a composition of a translation in $\mathrm{Sol}_0$ followed by an element of $G_0$. In particular,
\[
\mathrm{Isom}(\mathrm{Sol}) \cong \mathrm{Sol} \rtimes D_4.
\]

 In the sequel,  the following matrices in $\GL(2,\ZZ)$ are used:
$$S=\bpmat 0 & 1\\1 & 0\epmat\quad  L=\bpmat 1 & 1\\ 0 & 1\epmat\quad \quad R=SLS=\bpmat 1 & 0\\1 & 1\epmat \quad A=\bpmat a & b\\ c & d\epmat\quad P=\bpmat 1 & 0\\ 0 & -1\epmat.$$
The notation refers to Swap, Left, Right, Arbitrary, and Parity.

A torus bundle over the circle, $M_A$, has $\SOL$ geometry when $A$ is Anosov, that is to say when $A$ has two real eigenvalues that are not $\pm1$. We refer to such manifolds as $\SOL$ bundles. If $\det A = 1$, then $A$ is Anosov if and only if $|\trace(A)|>2$. When $\det A = -1$, the condition that $A$ is Anosov is equivalent to $\trace(A) \ne 0$. 
 
 A closed 3-manifold $M$ admits a $\SOL$ geometry if and only if  $M$ is homeomorphic to $M_A$ as above, or $M$ has a 2-fold cover by such a bundle. In the latter case, $M$ is a  semi-bundle that is the union of two orientable twisted $I$-bundles over a Klein bottle. These semi-bundles are also called \emph{sapphire manifolds} in the literature. We now describe this in more detail.

Let $X=\RR^2\times[-1,1]$ and let $\Gamma\subset\Aut(X)$ be the group generated by
\[
\alpha(x,y,z) = \left(x+\tfrac{1}{2}, -y, -z\right), 
\quad 
\beta(x,y,z) = (x, y+1, z).
\]
Then $Y=K\widetilde{\times}I\cong X/\Gamma$ is the orientable twisted $I$-bundle over a Klein bottle $K$. 
Indeed, $\alpha$ and $\beta$ preserve $\mathbb{R}^2 \times 0$, and the induced action on this slice gives
\[
(\mathbb{R}^2 \times 0)/\Gamma \cong K.
\]
 The boundary of $Y$ is a torus that is identified with the orientation double-cover of $K$
by projection in the $z$-direction. It is the quotient of $\RR^2\times\{\pm1\}$ by $\Gamma$.

Let $\Gamma'$ be the subgroup of elements of $\Gamma$ that preserve the orientation of $\RR^2\times 0$. Then $\Gamma'$ has index $2$ in $\Gamma$, and  $\Gamma'\cong\ZZ^2$ is generated by $\alpha^2, \beta$. This defines a double cover $\pi_Y:\widetilde{Y}\rightarrow Y$ with $\widetilde{Y}=X/\Gamma'\cong \torus\times I$.
 
 There is a unique non-trivial covering transformation for $\pi_Y$, given by an involution 
$\overline{\alpha}$ of $\widetilde{Y}$ that is covered by  $\alpha$. It interchanges the two boundary components of $\torus\times [-1,1]$. The restriction of $\pi_Y$ gives  the identifications  \beq{KxI}\torus\equiv \torus\times1 \equiv(\RR^2\times1)/\Gamma'\equiv \bdy Y. \eeq

The mapping class group of the Klein bottle is the dihedral group of order $4$. There are  two (canonical) isotopy classes of unoriented simple closed curves on a Klein bottle that are preserved by this group, corresponding to $\alpha$ and $\beta$. This gives a canonical basis of $H_1(\bdy Y)$ described below.

Using the natural identification of $\Gamma'$ and $H_1(\partial Y;\ZZ)$, it follows that 
$(\alpha^2,\beta)$ is a basis of $H_1(\partial Y;\ZZ)\cong \ZZ^2$. The action of ${\overline{\alpha}}_*$ on $H_1(\widetilde{Y})$ is given, in this basis, by the parity matrix
$P$. The cyclic subgroups
generated by $\alpha^2$ and $\beta$ are the unique cyclic subgroups preserved by this action. Therefore, the unoriented isotopy classes of these elements are uniquely
determined by $K$. This gives four {\em canonical bases}  $(\pm\alpha^2,\pm\beta)$  of $H_1(\bdy Y)$, which lift to bases of $H_1(\torus \times -1)$ and $H_1(\torus \times 1)$. Using these lifted bases, the matrix of $\overline \alpha_*$ is $P$.

Fixing a canonical basis, a mapping class of $\partial Y \equiv \torus$ defines a matrix $A \in \GL(2,\ZZ)$ and vice versa. Notice that the diagonal matrices in $\GL(2,\ZZ)$ correspond  to the mapping classes of $\partial Y$ that interchange canonical bases. These mapping classes of $\partial Y$ extend to $Y$.

Let $N_A$ be the result of gluing two copies of $Y$, denoted by $Y_1$ and $Y_2$, along their boundary tori via a mapping class $\phi$ that is represented by $A \in \GL(2,\ZZ)$ with respect to the fixed canonical bases for $\partial Y_1$ and $\partial Y_2$. The resulting manifold $N_A$ is a torus semi-bundle with gluing map $A$. The 2-fold cover $\torus \times I \to K \wtimes I$ extends to a 2-fold cover $M_{A^{-1}PAP} \to N_A$, which is referred to as the \emph{torus bundle double cover} of the semi-bundle.

Swapping the roles of $Y_1, Y_2$ shows that $N_A \cong N_{A^{-1}}$. Additionally, precomposing the gluing map $\phi$ with automorphisms of $Y_1$ and $Y_2$ that restrict to diagonal matrices on $\partial Y_1, \partial Y_2$, yields a manifold homeomorphic to $N_A$. 
It follows that $N_A \cong N_{B}$ if $B = DA^{\pm 1}E$ where $D,E \in \GL(2,\ZZ)$ are diagonal. 
In fact, Morimoto showed that the converse is also true \cite[Theorem 1]{MR811801} giving the classification of all torus semi-bundles. Also see \cite[Theorem 2.8]{Hatcher} and \cite[Theorem 2.4]{MR2666128}. In the particular case of $\SOL$ semi-bundles, the classification is given by \cite[Proposition 1.5]{MR2666128}. These results can be summarized as:

\begin{proposition}\label{solsemibundle} Two torus semi-bundles $N_A$ and $N_B$ are homeomorphic if and only if $B=DA^{\pm1}E$, where $D, E \in \GL(2,\ZZ)$ are diagonal. If $A\in\GL(2,\ZZ)$ then the torus semi-bundle $N_A$ admits a $\SOL$ structure if and only if every entry of $A$ is non-zero.
\end{proposition}

\noindent Using the classification, the characterization of Sol semi-bundles can be refined as follows. 

\begin{theorem}\label{funfact}
 The map defined by $\vartheta(A)= N_A$ is a bijection from $\{A\in\GL(2,\ZZ)\ :\ a \geq d\ \  \textnormal{and}\ \ b,c,d>0\}$ onto homeomorphism classes of $\SOL$ semi-bundles.
\end{theorem}

\begin{proof} 
Suppose $N_A$ is $\SOL$. Then $abcd\ne 0$, so the integers $ad$ and $bc$ are non-zero.  Moreover, $ad$ and $bc$ have the same sign since $\det A=ad-bc =\pm1$. Thus, $A$ has an even number of negative entries. After  replacing $A$ by $DAE$ for an  appropriate choice of diagonal matrices, it follows that $a,b,c,d$ are all strictly positive.  
Given that $N_A\cong N_{A^{-1}}$, $\det A = \pm1$, and 
$$\bpmat a & b\\ c & d\epmat ^{-1} =  \pm\bpmat -1 & 0\\ 0 & 1\epmat \bpmat d & b\\ c & a \epmat\bpmat -1 & 0\\ 0 & 1\epmat,$$
 we may assume $a \geq d$. It follows that $\vartheta$ is surjective. It is easy to show that $\vartheta$ is injective so we leave it to the reader. \end{proof}

\section{Regular Languages}\label{Regsec} In this section, regular languages are used to encode matrices in $\GL(2, \ZZ)$ that define $\SOL$ bundles and semi-bundles. A \emph{string} is a finite sequence of zero or more elements of a set called the \emph{alphabet}.  A \emph{regular language} is a set of strings obtained using only the following operations. 

Every finite set of strings is a regular language. If $\Rcal$ and $\Rcal'$ are regular languages then $\Rcal\Rcal'$ is the regular language consisting of a string in $\Rcal$ followed by one in $\Rcal'$. In addition, $(\Rcal|\Rcal')$ is the regular language $\Rcal\cup \Rcal'$, and $\Rcal^*$ is the regular language consisting of the concatenation of zero or more strings in $\Rcal$.

\Cref{matrix lemma} below provides a characterization of Anosov matrices, and of matrix representatives for $\SOL$ semi-bundles, using regular languages that will later be encoded by labeled train tracks. 
In particular, a regular language is given for each set of matrices.
 Define 
 $$\Pcal_0=(L|S)^* ,\qquad \Pcal= (L|S)^* LSL(L|S)^*,\qquad \Pcal_a=( L|LS)^* LS.$$  
  Regarding $L,S$ as letters of an alphabet and not as matrices, each of these sets of words is a regular language with alphabet $\{L,S\}$. The notation suggests $\Pcal$ is positive, $\Pcal_0$ is positive or zero, and $\Pcal_a$ is {\bf A}nosov. 
  
 Each of these sets determines a subset of $\GL(2,\ZZ)$ using the matrices above.
Since $S^2=I$ the set  of matrices for $\Pcal_a$ can be represented by all products of $L$ and $S$ that start with $L$ and end with $LS$, and  the set of matrices for
 $\Pcal$ consists of all products of $L$ and $S$ that contain the string $LSL$.

\begin{lemma}\label{matrix lemma} If $A\in\GL(2,\ZZ)$, then\\
(1) $\Pcal_0$ is the set of non-negative matrices.\\
(2) $\Pcal$ is the set of strictly positive matrices.\\
(3) $A$ is Anosov if and only if it is conjugate into $\pm\Pcal_a$.
\end{lemma}
\begin{proof}   (1) It is clear that every matrix in $\Pcal_0$ has non-negative entries.  
By the Euclidean algorithm, $L$ and $R$ generate the free monoid $X_+\subset \SL(2,\ZZ)$ that consists of all matrices with non-negative entries and determinant 1 (see, for instance, Remark 4.7 in \cite{MR4699874}). Since $L,S\in\Pcal_0$ and $R=SLS$,
it follows that $X_+\subset \Pcal_0$. Let $X_-\subset\GL(2,\ZZ)$ be the set of non-negative matrices 
with determinant $-1$. Since $\det S=-1$ then $S X_-\subset X_+$. Also $S^2=I$ so $X_-\subset SX_+\subset \Pcal_0$.  
Hence $X_+\cup X_-\subset \Pcal_0$.
This proves (1).

\gap 

 \noindent  (2) Every  word in $\Pcal$ contains $LSL$, and this
 is a strictly positive matrix. Since $L$ and $S$ are non-negative, it follows that all members of $\Pcal$ are strictly 
positive. For the reverse containment, note that strictly positive matrices are a subset of $\Pcal_0$. Given that $S^2=I$, a word is {\em reduced} if it does not contain $S^2$.   Every reduced word in $S$ and $L$ that does not contain $ L  S L $ is one of $S, L ^n,S L ^n, L ^nS$ or $S L ^nS$ for $n\ge1$. These matrices all have a zero entry. This proves (2).

\gap

\noindent  (3)  
 The property that a  matrix $A$ is Anosov is preserved by conjugation. Additionally, $A$ is Anosov if and only if $\pm A^n$ is Anosov for  $n\ne 0$.
If $A\in\Pcal_a$ is not strictly positive, then $A$ is of the form $$L^nS = \bpmat n & 1 \\ 1 & 0 \epmat,\qquad n>0$$ and $A^2$ is strictly positive.
Therefore, if $A\in\Pcal_a$, then $A^4$ is a strictly positive matrix with all entries greater than $1$. It follows that $\tr (A^4)>2$, which implies that $A^4$ is Anosov, and hence $A$ is Anosov.
Thus, every element of $\pm\Pcal_a$ is Anosov.

The idea for the converse is to first reduce to the case that $A$ is non-negative. 
 
\gap

\noindent \textbf{Claim:} If $A$ is Anosov, then $A$ is conjugate into $\pm \Pcal_0$. 

\gap

\noindent \textbf{Proof of Claim:} 
  If some entry of $A$ is zero, then since $A$ is Anosov, the zero entry is on the diagonal. By conjugating by $S$ as needed, we may assume that the zero entry is the $(2,2)$ entry.  After replacing $A$ by $-A$ if needed, we may assume that either $A$ is in $\Pcal_0$, or $A$ is one of
$$  \bpmat n &  -1\\  -1 & 0\epmat, \quad \bpmat n & - 1\\  1 & 0\epmat,\quad \bpmat n & 1 \\ -1 & 0\epmat,\qquad \qquad \textnormal{ for } n>0.$$
 Conjugating the first of these gives
$$\bpmat 0 & -1\\1 & 0\epmat \bpmat n & -1\\ -1 & 0\epmat\bpmat 0 & 1\\-1 & 0\epmat =\bpmat 0 & 1\\1 & n\epmat\in\Pcal_0.$$
In the remaining two cases, $\det A=1$, and $A$ is Anosov, so $n\ge 3$. Conjugating the first of these two  gives
$$\bpmat 1 & -1\\1 & 0\epmat \bpmat n & -1\\ 1 & 0\epmat\bpmat 1 & -1\\1 & 0\epmat^{-1} =  \bpmat 1 & n-2\\1 & n-1\epmat\in\Pcal_0.$$
The last matrix is conjugated into $\Pcal_0$ similarly. 
 Thus, if some entry is $0$, then $A$ is conjugate into $\pm\Pcal_0$.

Now, suppose that no entry of $A$ is zero. Since $\det A=ad-bc =\pm1$, $abcd\ne 0$, and $A$ is integral, it follows that $A$ has an even number of negative entries. Therefore, $\pm A \in \pm \Pcal_0$ or $A$ has two positive entries and two negative entries. The result when the negative entries are either both on or both off the diagonal follows from conjugation by $P$.
This leaves the case where the two negative entries are either in the same row or in the same column.

Consider the case
$$A=\bpmat a & b\\ -c & -d\epmat$$
with all $a,b,c,d>0$. Additionally, assume for now that $a \geq b$. Then
$$\bpmat 1 & 0\\1 & 1\epmat \bpmat a & b\\ -c & -d\epmat \bpmat 1 & 0\\-1 & 1\epmat=\bpmat a-b & b\\ a-c+d-b & b-d\epmat=F$$ 
Since $F$ is Anosov, if some entry of $F$ is zero or if every entry is strictly positive, the claim follows by the argument above. 
Thus, we may assume $F$ has no zero entries and an even number of negative entries. Since $a-b\ge0$, it follows that $a-b>0$. Additionally $b >0$ by assumption. Thus, $F$ has two negative entries, so
$b-d<0$ and
$a-c+d-b<0$. Then
$$ \|F\|=(a-b)+b+(c-a+b-d)+(d-b)=c<a+b+c+d=\|A\|.$$ Induction
on $\|A\|$ proves the claim in this case.

The other cases, when the negative entries of $A$ are either in the first row or else in the same column and where $a \geq b$,
are proved similarly. If $a<b$, then $c\le d$ because $\det A=\pm1$. In this case, conjugating $A$ by $S$ reduces to one of the previous cases where $a\ge b$. For example, if $A$ is as above, then 
$$ SAS^{-1} = S\bpmat a & b \\ -c & -d\epmat S = \bpmat -d & -c\\ b & a\epmat.$$

\noindent This proves the claim that $A$ is conjugate into $\pm \Pcal_0$. \hfill $\blacksquare$

\gap

Suppose $A$ is Anosov. By the remarks at the start of the proof and along with the claim, we may assume that $A\in\Pcal_0$. If $A$ is strictly positive, then by (2) it is represented by a word in $(L|S)^*$ that contains the subword $LSL$. Thus, it is conjugate to a positive word that starts with $L$ and ends with $LS$,
and therefore conjugate into $\Pcal_a$. Otherwise,  if $A$ is not strictly positive, then some entry is zero. Since $A$ is Anosov, the zero entry of $A$ is on the diagonal and $A$ is one of
 $L^nS$ or $SL^n$ for some $n>0$. But $SL^n$ is conjugate to $L^nS$, and $L^nS\in\Pcal_a$. Hence, every Anosov matrix is conjugate into $\pm\Pcal_a$.
\end{proof}

The theorem below follows immediately by \Cref{matrix lemma}
 and the surjectivity of $\vartheta$ from \Cref{funfact}.
 
\begin{theorem}\label{solcor} Let $[M]$ denote the homeomorphism class of the manifold $M$. \\
(1) $\{[M_{A}]: A\in\pm\Pcal_a\}$ is the set of all  homeomorphism classes of  $\SOL$ bundles.\\
(2)  $\{[N_{A}]: A\in \Pcal\}$ is the set of all homeomorphism classes of $\SOL$ semi-bundles. 
\end{theorem}

\section{Main Theorem}\label{MainThmsec} In this section it is important to keep in mind that if $\omega$ is an immersed path in a labeled train track, then by \Cref{lem: immersions from pathsa} the holonomy, denoted by $A(\omega)$, is the product of the inverses of the edge labels in the order in which the edges are traversed (see \Cref{def: pathprod}).

  The train tracks $\EuG$ and $\overline\Gamma$ below are \emph{orientable} in the sense that every edge has an orientation (different from the arrows depicting labels) so that every chart is orientation preserving. It follows that an immersed loop in such train tracks will always cross a particular edge in the same direction.

\begin{theorem}\label{thm: sol bundles}
   The set of $\SOL$ bundles is LCD. 
\end{theorem}

\begin{proof}
\Cref{sol bundles} shows a labeled train track $\EuG$. Let $\Wcal_{\EuG}$ be the branched manifold associated to $\EuG$. By \Cref{equiv thm}, it suffices to show that a closed 3-manifold $M$ admits a proper PL immersion into $\Wcal_\EuG$ if and only if $M$ is a $\SOL$ bundle. 

Suppose that $M$ is a closed $3$-manifold and
$\Omega : M \to \mathcal{W}_{\EuG}$
is a proper PL immersion. By \Cref{prop:train track torus bundle}, the manifold $M$ is a torus bundle over $\mathbb{S}^1$ and there is an immersion loop  $\omega:\SS^1\rightarrow \EuG$ for $\Omega$. Fix $x \in \EuG$ so that $\omega$ is based at $x$. By \Cref{lem: immersions from pathsa}, $(\mon(\omega); \widetilde{x}, \widetilde{x}) = A(\omega) = A$ for some $A \in \GL(2, \QQ)$.  From \Cref{thm: immersions from paths}, it follows that the monodromy of $M$ is conjugate to $A$ in $\GL(2, \QQ)$. Hence $\det (A)=\pm1$.
\begin{figure}[h!]
\vspace{0.25 cm}
\begin{overpic}[scale=0.7]{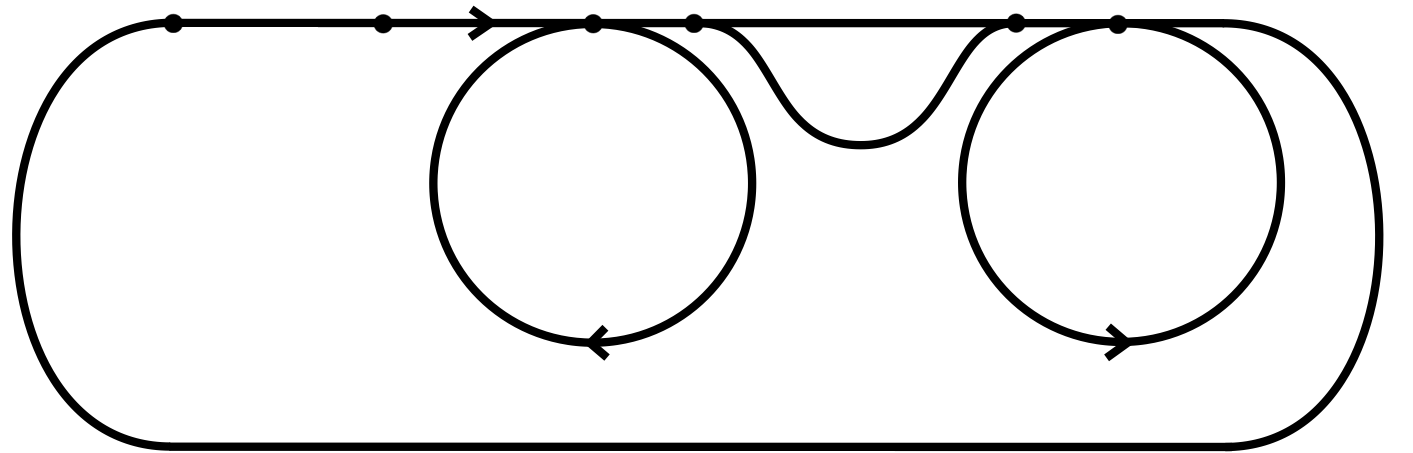}
    \put(56,17){$-I$}
    \put(33,32.5){$L$}
    \put(19,32.5){$S$}
     \put(40,10){$2L$}
     \put(77,10){$2I$}
     \put(11.5,32.5){\color{blue}$y$}
\end{overpic}
 \caption{$\EuG$} 
 \label{sol bundles}
\end{figure}
The orientation on the train track $\EuG$ implies that if $\omega$ traverses the $2L$ loop positively (resp. negatively) then it traverses the $2I$ loop negatively (resp. positively). 
Since $\det A=\pm1$ and all other edge labels have determinant $\pm 1$, it follows that $\omega$ traverses the loop labeled by $2L$ and the loop labeled by $2I$ the same number of times. This means that $\omega$ must contain $y$. 

Thus, we can choose $\omega$ to be a loop based at $y$ and oriented so that it traverses the edges labeled $S,L$ in that order. It follows that, as a word, $A$ starts with $S^{-1}L^{-1}$. Then, the remaining edges contribute products of $(2L)^{-1},2I$, and $-I$ to $A(\omega)$.

As a result, $A^{-1}$ or $-A^{-1}$ is a reduced word in $L$ and $S$ that begins with $L$ and ends with $LS$. As noted above, any such reduced word is an element of $\Pcal_a$, so that $A^{-1} \in \pm \Pcal_a$. By \Cref{thm: immersions from paths}(4), the bundle $M_{A}$ immerses into $\Wcal_{\EuG}$ with the same immersion path $\omega$. Therefore, $M$ is virtually equivalent to $M_{A}$ by \Cref{thm: immersions from paths}(1). Moreover, \Cref{solcor} implies $M_{A^{-1}}$ is a Sol bundle, and since $M_A \cong M_{A^{-1}}$, the torus bundle $M_{A}$ is Sol as well. Since the property of being Anosov is invariant under $\GL(2, \QQ)$ conjugation, $M$ is a $\SOL$ bundle.

On the other hand, suppose $M$ is a $\SOL$ bundle. By \Cref{solcor}, $M \in [M_A]$ where $A \in \pm \Pcal_a$. Again using the fact that $S^2 =I$, the matrix $A$ can be written in the form  $$\pm \, L^{n_k}S \cdots L^{n_1}S,$$ where $n_i \in \NN$ for all $i$. By induction, there is an immersion $\omega: \SS^1 \to \EuG$ based at $y$ with monodromy $A^{-1}$. Then, \Cref{thm: immersions from paths} gives the desired immersion of $M_{A^{-1}} \cong M_A$ into $\Wcal_\EuG$. \end{proof}

The loop in $\EuG$ is labeled $2L$ rather than $L$ since, otherwise, the non-$\SOL$ manifold $M_L$ would immerse into $\Wcal_{\EuG}$. This trick to eliminate unwanted immersions will be used frequently in subsequent papers.

As an aside, note that the proof of \Cref{thm: sol bundles} uses immersion paths that traverse the $2L$ and $2I$ loops of the train track $\EuG$ the same number of times. These immersion paths correspond to the subset of the regular language $\overline{\Pcal_a} = (2I)^*(2L|LS)^*LS$ consisting of those strings that map to zero under the map $F:\overline{\Pcal_a}\rightarrow\ZZ$ defined by $F(2L)=1, F(2I)=-1, F(LS) = 0$ and $F(w_1w_2)=F(w_1)+F(w_2)$. Subsets of this form are well-studied in the theory of regular languages, and $F_0 =\{w \in \overline{\Pcal_a} \  | \ F(w) = 0\}$ is an example of a \emph{context-free language}, which are languages recognized by a \emph{pushdown automaton}, see \cite[Theorem 2.20]{sipser2013introduction}.

\subsection{Branched Semi-bundles} A torus semi-bundle is the quotient of a torus bundle by a free involution that sends torus fibers to torus fibers and so that the quotient contains two Klein bottles. This idea can be extended to branched torus bundles as follows.

\Cref{solbrmfdcover} shows a labeled train track $\overline{\Gamma}$ that determines a branched manifold $\Wcal_{\overline{\Gamma}}$ with projection $\overline{p}:\Wcal_{\overline{\Gamma}}\rightarrow\overline{\Gamma}$ which has torus fibers.
There is a free involution $\sigma$ of $\Wcal_{\overline{\Gamma}}$ that sends torus fibers to torus fibers and the quotient is a universal branched manifold $\Wcal_{\SOL}$ for $\SOL$ manifolds, as we will now show.

There is a decomposition $\overline{\Gamma}=J_1\cup J_2\cup J_3\cup J_4$ as the union of two arcs $J_1=x'x$ and $J_3=ww'$ both labeled by $P$, and two isomorphic labeled train tracks $J_2\cong J_4$. Let $\Wcal_i=\overline{p}\,^{-1}(J_i)$. Then $\Wcal_{\overline{\Gamma}}=\Wcal_1\cup\Wcal_2\cup\Wcal_3\cup\Wcal_4$ and the branched manifolds $\Wcal_i$ have disjoint interiors. There are identifications $\Wcal_1\cong \torus\times I\cong\Wcal_3$ and $\Wcal_2\cong\Wcal_4$, where the branched structure on each $\Wcal_i$ is specified by the Gluing Theorem \cite[Theorem 5.13]{CPT}.

\begin{figure}
\raggedleft
 \vspace{0.25 cm}
\begin{overpic}[scale=0.44]{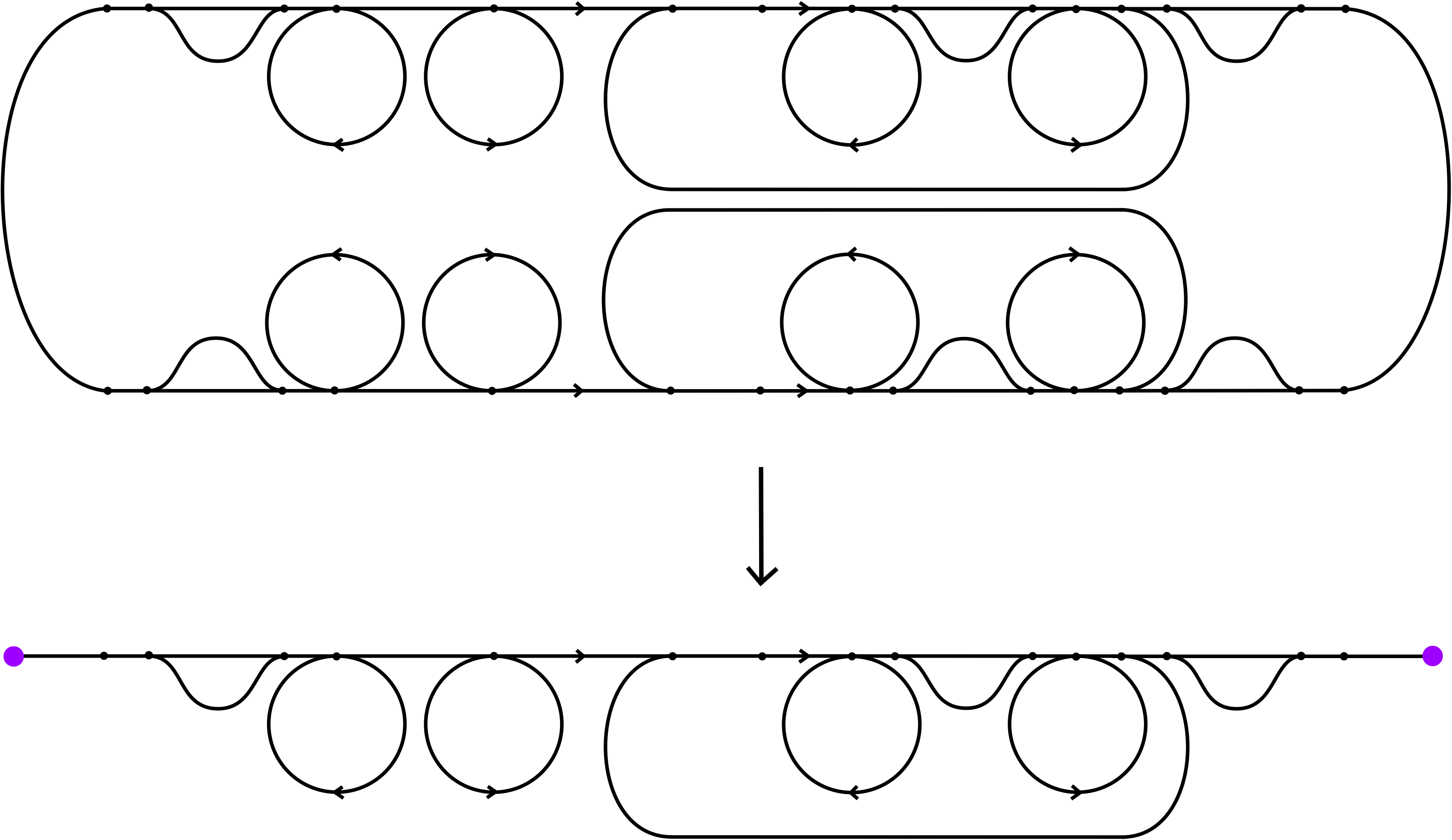}
\put(-10.75,42){\Large$\overline{\Gamma}$}
\put(-10.75,11.5){\Large$\Gamma$}
\put(14.5,58){$S$}
  \put(64,28){$-I$}
    \put(39,28){$L$}
    \put(49,28){$S$}
     \put(54.5,28){$L$}
     \put(54,22){$q$}
     \put(84,28){$S$}
     \put(22,37){$2L$}
     \put(32.5,36.5){$2I$}
     \put(57,37){$2L$}
     \put(72.5,37){$2I$}
     \put(64,58){$-I$}
     \put(14,28){$S$}
    \put(39,58){$L$}
    \put(49,58){$S$}
     \put(54.5,58){$L$}
     \put(84,58){$S$}
     \put(22,49){$2L$}
     \put(32.75,49){$2I$}
     \put(57.5,49){$2L$}
     \put(73,49){$2I$}
     \put(2,43){$P$}
     \put(96.5,43){$P$}
      \put(7,58.5){\color{blue}$x$}
     \put(91,58.5){\color{blue}$w$}
     \put(7,28){\color{blue}$x'$}
     \put(91,28){\color{blue}$w'$}
     \put(64,13.5){$-I$}
    \put(39,13.5){$L$}
    \put(14.5,13.5){$S$}
    \put(49,13.5){$S$}
     \put(54.5,13.5){$L$}
     \put(85,13.5){$S$}
     \put(32.5,4.5){$2I$}
     \put(57.5,4.5){$2L$}
     \put(73,4.5){$2I$}
     \put(22,4.5){$2L$}
     \put(-3,12.5){$y_1$}
     \put(100.5,12.5){$y_2$}
     \put(7,14){\color{blue}$\overline{x}$}
     \put(91,14){\color{blue}$\overline{w}$}
\end{overpic}
\caption{The labeled train track  $\overline{\Gamma}$ and its quotient $\Gamma$.} \label{solbrmfdcover}
\end{figure}
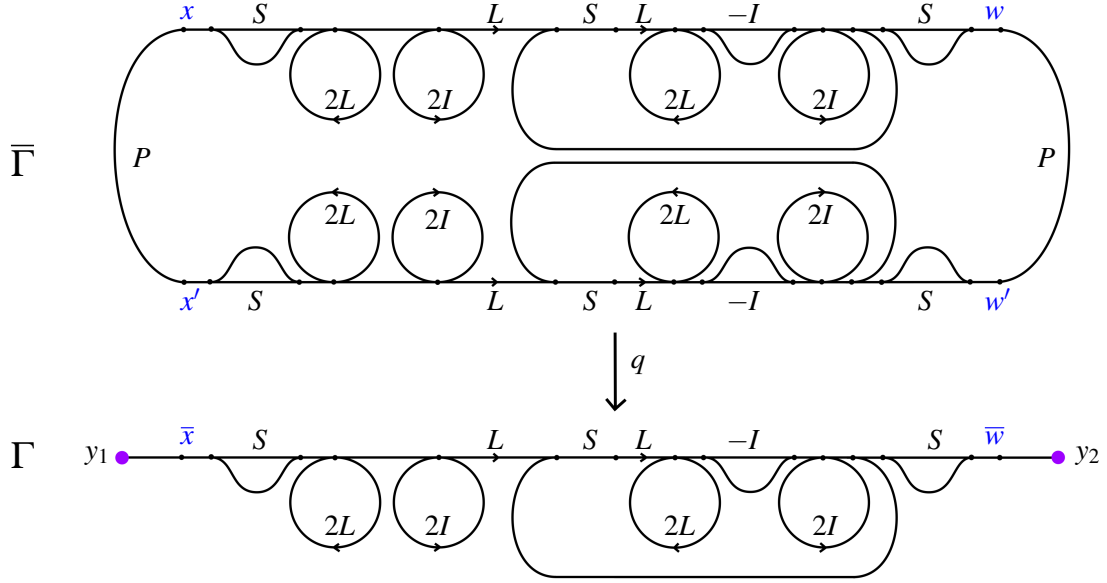

There is an involution
$\sigma$ on $\Wcal_{\overline{\Gamma}}$ with the following properties. Using the identification 
$\Wcal_1\cong \torus\times I$ with the double cover of $Y=K\widetilde{\times}I$, then $\sigma|_{\Wcal_1}=\overline{\alpha}$, the involution of $T \times I$ defined in \Cref{SOLsec}. 
Similarly, define $\sigma|_{\Wcal_3}$. Then $\Wcal_1/\sigma\cong Y\cong\Wcal_3/\sigma$. Define $\sigma|_{\Wcal_2}$ as the identification  $\Wcal_2\cong\Wcal_4$. Then $$\Wcal_{\SOL}=(\Wcal_{\overline{\Gamma}})/\sigma\cong Y_1\cup\Wcal_2\cup Y_2\qquad Y_1\cong K\widetilde{\times} I\cong Y_2.$$
The quotient map is a  $2$-fold covering  $\zeta:\Wcal_{\overline{\Gamma}}\rightarrow\Wcal_{\SOL}$ with deck transformation $\sigma$. 

The fact that $\Wcal_2 \cong \Wcal_4$ implies that for each chart $\phi$ of $\Wcal_{\overline{\Gamma}}$, the map $\phi \circ \sigma$ is also a chart for $\Wcal_{\overline{\Gamma}}$. That is to say, the atlas for $\Wcal_{\overline{\Gamma}}$ is invariant under deck transformations. This implies that the charts for $\Wcal_{\overline{\Gamma}}$ can be pushed down to a set of charts for $\Wcal_{\SOL}$ so that $\Wcal_{\SOL}$ is a branched manifold.

 The involution $\sigma$ of $\Wcal_{\overline{\Gamma}}$ covers 
 an involution of $\overline{\Gamma}$ that swaps  $J_2$ and $J_4$, and the quotient is the labeled train track $\Gamma$ shown in \Cref{solbrmfdcover}. 
 
 There is a map $p:\Wcal_{\SOL}\rightarrow\Gamma$ that is covered by $\overline{p}$ with the property that, except at the endpoints $y_1,y_2$, the map is the usual projection defined for a branched manifold associated to a labeled train track as in \Cref{def:train track mfd}.
The preimage of each endpoint $y_1, y_2$ is a  
Klein bottle and the preimage of a small neighborhood of each endpoint is a copy of $Y$. 
Let $q:\overline{\Gamma}\rightarrow\Gamma$ be the quotient map, then $p\circ\zeta=q\circ\overline{p}$.

As discussed at the start of \Cref{SOLsec}, the canonical bases on $\bdy Y$ lift to bases on $\torus\times -1$ and  
$\torus\times 1$, and these torus slices are mapped to each other by the matrix $P$ with respect to any of these lifted bases. This explains the matrix label $P$ on $J_1$ and $J_3$ in $\overline{\Gamma}$.

\begin{theorem}\label{thm: sol semi-bundles} A closed 3-manifold $M$ properly immerses into $\Wcal_{\SOL}$ if and only if $M$ admits a $\SOL$ structure. Thus, the set of $\SOL$ manifolds is LCD.
\end{theorem}
\begin{proof}  The first step is to show that every $\SOL$ manifold immerses into $\Wcal_{\SOL}$.
 Since $\Gamma$ contains a copy of $\EuG$, 
every $\SOL$ torus bundle immerses into $\Wcal_{\SOL}$.

Suppose that $N_A$ is a $\SOL$ semi-bundle. We may assume that  $A\in\pm\Pcal$.
Now $\Pcal$ consists of all strings in $S$ and $L$ that contain the substring $LSL$. An immersed path in $\Gamma$ that starts at $\overline{x}$ and ends at $\overline{w}$
produces a word in $S$ and $L$ that contains $LSL$ and a numerical factor of  $\pm 2^{\pm m}$ that depends on the number of times the path traverses the $2L, 2I$ loops and the $-I$ arc. All words in $\pm \Pcal$ with no consecutive occurrences of $S$ arise in this way.
 The loops labeled $2L$ and $2I$ in $\Gamma$ appear in pairs. By the orientability of the train track between $\overline{x}$ and $\overline{w}$, traversing the loops labeled $2I$ the same number of times as their corresponding $2L$ loops guarantees that $m=0$. Using the fact that $S^2 =I$ in $\GL(2, \ZZ)$, such paths give all matrices corresponding to words in $\pm\Pcal$, and in particular the matrix $A \in \pm \Pcal$ we began with.

 Thus, there is a path $\omega$ in $\Gamma$ that starts at $\overline{x}$ and ends at $\overline{w}$ so that if $\widetilde{x}$ is a lift of $\overline{x}$ to the universal cover, then $(\hol(\omega); \widetilde{x}, \widetilde{x})$ is $A^{-1}$. Note that $\omega$ can be chosen so that it traverses the $2L, 2I$ loops in such a way that the immersion $\Omega: \torus \times I \to \Wcal_{\SOL}$ produced by \Cref{thm: immersions from paths}(3a) satisfies the hypothesis of \Cref{thm: immersions from paths}(3b), and thus $n=1$.

There is a decomposition $$N_A=Y_1\cup (\torus\times I)\cup Y_2$$
where $Y_1$ and $Y_2$ are copies of $Y$. Fix the canonical basis $(\alpha^2, \beta)$ for each copy. Then, $N_A$ is obtained by identifying the canonical basis of $\bdy Y_1$ with the marked basis of $\torus\times 0$, and the canonical basis of 
$\bdy Y_2$ with the image of the marked basis for $\torus\times 1$ under $A$. 

Let $Y_1'$ be the result of identifying the canonical basis of $\partial Y_1$ and the marked basis of $T \times 0$. The fact that $\Omega$ satisfies \Cref{thm: immersions from paths}(3b) and, in particular, that $n=1$ implies that $\Omega$ extends over $Y_1$ to an immersion of $Y_1'$ into $\Wcal_{\SOL}$. Moreover, since $(\hol(\omega); \widetilde{x}, \widetilde{x}) = A^{-1}$, the map $(\Omega|_{T \times 1})_*$ is represented by the matrix $A$ with respect to the marked bases for the domain and codomain, so $\Omega$ further extends to an immersion of $N_A$ into $\Wcal_{\SOL}$. It follows that every $\SOL$ semi-bundle immerses into $\Wcal_{\SOL}$.

For the converse, first note that there are immersions $M\rightarrow\Wcal_{\SOL}$ whose immersion paths contain interior points mapping to endpoints of $\Gamma$. For example, if one starts with
an immersion of a semi-bundle $\Omega:M\rightarrow\Wcal_{\SOL}$ as described above, and $f:\overline{M}\rightarrow M$ is the torus bundle double cover, then $\Omega\circ f$ has this property. 

Now, suppose $\Omega:M\rightarrow\Wcal_{\SOL}$ is an immersion. Since $\Wcal_{\overline{\Gamma}}$ is a 2-fold cover of $\Wcal_{\SOL}$,
there is a 2-fold cover of $M$ that immerses into $\Wcal_{\overline{\Gamma}}$. Let $\overline{M}$ be a connected component of this cover of $M$ and let $\overline{\Omega}: \overline{M} \to \Wcal_{\overline{\Gamma}}$ be the restriction of the immersion. We will prove that every manifold that immerses into $\Wcal_{\overline{\Gamma}}$ is $\SOL$ so that $\overline M$ is $\SOL$. A  $3$-manifold finitely covered by a $\SOL$ manifold is also a $\SOL$ manifold, which then implies that $M$ is a $\SOL$ manifold as well.

Let $\overline{\Omega}: \overline{M} \to \Wcal_{\overline{\Gamma}}$ be the immersion defined above with immersion loop $\overline{\omega}$. Then $\overline{\omega}$ traverses the $2L$ loops and  the $2I$ loops in $\overline\Gamma$ the same number of times because  $\det A(\overline{\omega})=\pm 1$. For this reason, we ignore the numerical factor $2^{\pm m}$ in what follows.

Recall that there is a decomposition $\overline{\Gamma}=J_1\cup J_2\cup J_3\cup J_4$. The product of matrices (in reverse order) along an immersed path in $J_2$ starting at $x$ and ending at $w$ is given by a word in $\pm\Pcal$.  An immersed path in $J_3\cup J_4\cup J_1$ starting at $w$ and ending at $x$ gives a word of the form $P A^{-1}P$ for some $A \in \pm \Pcal$. 

Recall that $S$ is its own inverse, and observe that $P=P^{-1}$, $PSP=-S$, and $PL^{-1}P=L$. That is to say, conjugation by $P$ sends $L^{-1}$ to $L$ and $S$ to $-S$. It follows that the involution $\tau(A)=PA^{-1}P$ satisfies  $\tau(\pm\Pcal)=\pm\Pcal$ when considering the words in $\Pcal$ as matrices. Thus, an immersed path in $J_3\cup J_4\cup J_1$ starting at $w$ and ending at $x$ gives a word in $\tau(\pm\Pcal)=\pm\Pcal$.

  Now, if the immersion loop $\overline{\omega}$ in $\overline{\Gamma}$ does not meet $x$, then its image is contained in one of the two copies of $\EuG$ in $\overline{\Gamma}$, so $\overline{M}$ is $\SOL$. If $\overline{\omega}$ meets $x$, then it alternates between traversing a path in $J_2$ in the direction from $x$ to $w$ and traversing a path in $J_3\cup J_4\cup J_1$ in the direction from $w$ to $x$. This gives a word  in   the regular language
$$\widetilde{\Pcal} = \pm (\tau(\Pcal)\Pcal)^*\tau(\Pcal)\Pcal.$$
Since the product of strictly positive matrices is a strictly positive matrix and $\tau(\Pcal) \subset \pm \Pcal$, it follows that $\widetilde{\Pcal} \subseteq \pm \Pcal$. Given the additional fact that any word in $\Pcal$ is conjugate into $\Pcal_a$, it follows that $\overline{M}$ is $\SOL$. \end{proof}

\noindent \Cref{thm: sol bundles} and \Cref{thm: sol semi-bundles} give \Cref{solbrmfdthm}.
\gap

As a closing remark, notice that the branched manifold $\Wcal_{\SOL}$ has some special properties that do not hold for some other Thurston geometries. 
It has an affine structure in the sense
that there are affine structures on the sheets with the property that all the transition maps are affine.
This implies that there is a developing map $\widetilde\Wcal_{\SOL}\rightarrow \RR^3$ and  holonomy
homomorphism $$\pi_1(\Wcal_{\SOL})\rightarrow\operatorname{Aff}(\RR^3),$$ defined in an analogous way as for manifolds, where the charts for $\Wcal_{\SOL}$ replace the coordinate  charts.
Also, $\Wcal_{\SOL}$ is obtained from $\bigsqcup \torus\times I$ by gluing boundary components. The foliation
of each $\torus\times I$ by lines $x\times I$ gives a kind of 1-dimensional foliation of $\Wcal_{\SOL}$
by train tracks. The restriction of this foliation to each sheet is a foliation by lines. If $M$ is immersed in  $\Wcal_{\SOL}$,
this foliation pulls back to a $1$-dimensional foliation of $M$. 

In the bundle case, this foliation can be oriented and
then gives the Anosov flow for $M$. Of course, $M$ also has a decomposition into tori that is a fiber bundle
at the generic points, and this pulls back to a foliation of $M$ by tori.
Thus, one might regard $\Wcal_{\SOL}$ as having some kind of $\SOL$ geometry. In addition, it is easy to check that if $\Omega:M\rightarrow \Wcal_{\SOL}$ is an immersion, then it is $\pi_1$-injective.

\bibliography{solrefs.bib} 
\bibliographystyle{abbrv} 
\end{document}